\documentclass[12pt]{amsart}

\usepackage{amsfonts, amsmath, amsthm, empheq, amssymb, bm, color,physics,enumerate,verbatim,comment}
\allowdisplaybreaks[4]

\newtheorem{lem}{Lemma}[section]
\newtheorem{thm}{Theorem}[section]

\newtheorem{cor}{Corollary}[section]
\newtheorem{rem}{Remark}[section]

\def\R{\mathbb R}

\def\De{\Delta}
\def\al{\alpha}
\def\be{\beta}

\def\ep{\epsilon}
\def\ka{\kappa}

\def\f{\frac}

\def\Ga{\Gamma}
\def\Schr{Schr\"odinger}

\def\la{\lambda}

\def\na{\nabla}

\def\Om{\Omega}
\def\ov{\overline}
\def\pa{\partial}

\def\Im{\mathrm{\,Im\,}}

\newtheorem{theorem}{Theorem}[section]

\newcommand{\C}{\mathbb{C}}

\newcommand{\N}{\mathbb{N}}
\newcommand{\K}{\mathbb{K}}

\newcommand{\cB}{\mathcal B}

\newcommand{\cQ}{\mathcal Q}

\newcommand{\cH}{\mathcal H}

\newcommand{\cI}{\mathcal I}
\newcommand{\cS}{\mathcal S}
\newcommand{\cV}{\mathcal V}
\newcommand{\cG}{\mathcal G}
\newcommand{\cE}{\mathcal E}

\usepackage{times}

\usepackage{epsfig}

\title{Stability analysis of inverse problems for coupled magnetic Schr\"odinger equations}
\author{Mohamed Hamrouni}
\address{University of Sousse, ESSTHS , LAMMDA , Tunisia}
\email{hamrouni.mohamed4@gmail.com}
\author{Moez Khenissi}
\address{University of Sousse, ESSTHS , LAMMDA , Tunisia}
\email{moez.khenissi@essths.u-sousse.tn}
\author{\'Eric Soccorsi}
\address{CPT, Aix Marseille Univ, Universit\'e de Toulon, CNRS, Marseille, France}
\email{eric.soccorsi@univ-amu.fr}
\makeatletter
\def\section{\@startsection{section}{1}%
	\z@{.5\linespacing\@plus.7\linespacing}{.5\linespacing}%
	{\normalfont\bfseries\raggedright}}
\makeatother
\begin{document}
\maketitle

\bigskip
\bigskip
\begin{minipage}{1\textwidth}
\begin{abstract}
	\small We consider the inverse coefficient problem of simultaneously determining the space dependent electromagnetic potential, the zero-th order coupling term and the first order coupling vector of a two-state \Schr\ equation in a bounded domain of $\R^d$, $d \ge 2$, from finitely many partial boundary measurements of the solution. We prove that these $3d+3$ unknown scalar coefficients can be H\"older stably retrieved by $(3d+2)$-times suitably changing the initial condition attached at the system.	\\
		
	
\end{abstract} 
\end{minipage}
\bigskip
\bigskip

\section{Introduction}
\label{sec-multi}
The present article is concerned with the identification of the two magnetic Laplacians and the linear coupling of a two-state quantum system by knowledge of finitely many partial boundary observations of the solution. Namely, 
given $T \in (0,\infty)$ and a bounded domain $\Om$ 
in $\R^d$, $d\in \mathbb{N}=\{1,2,\ldots \}$, with boundary $\Gamma=\partial \Om$, we consider the following initial-boundary value problem (IBVP) in the unknowns $u^\pm(x,t)$, where $x \in \Om$ and $t \in (0,T)$, 
\begin{equation}
	\label{sys:ori}
	\left\{
	\begin{array}{ll}
		-i\partial_t u^+ -\Delta_{A^+} u^++q^+u^+ +\Phi\cdot\nabla u^-+\phi u^-=0&\textrm{in}\ Q=\Om \times (0,T) \\
		-i\partial_t u^- -\Delta_{A^-} u^-+q^-u^- -\Phi\cdot\nabla u^++\phi u^+=0&\textrm{in}\ Q \\
		u^+(\cdot,0)=u^+_0,\ u^-(\cdot,0)=u^-_0&\textrm{in}\ \Om\\
		u^+=g^+,\ u^-=g^-&\textrm{on}\ \Sigma= \Gamma \times (0,T),
	\end{array}
	\right.
\end{equation}
with initial state $u_0^\pm$ and non-homogeneous Dirichlet boundary condition $g^\pm$. Here, $q^\pm : \Om \to \Bbb{C}$ is a complex-valued electric potential and 
\begin{equation}
\label{def-deltaA}
\De_{A^\pm} = (\na + iA^\pm)\!\cdot\! (\na + iA^\pm) = \De + 2iA^\pm\!\cdot\!\na + i(\na\!\cdot\! A^\pm) - \abs{A^\pm}^2
\end{equation}
denotes the magnetic Laplace operator associated with the magnetic vector potential $A^\pm : \Om \to \R^d$.
The coupling between the two \Schr\ equations appearing in \eqref{sys:ori} is linear, the coefficient of the first order term being expressed by the vector potential $\Phi: \Om \to \R^d$, while the one of the zero-th order term is $\phi: \Om \to \Bbb{C}$. We refer the reader to \cite[Section 1]{KRS} for the physical relevance of the IBVP \eqref{sys:ori} for modeling the dynamics of two states quantum systems such as spin-$\frac{1}{2}$ particles, like electrons, subject to time-independent magnetic fields.\\
In this work we examine the stability issue  in the inverse problem of simultaneously determining the electromagnetic potentials $A^\pm$, the electric potentials $q^\pm$ and the coupling terms $(\Phi,\phi)$ from partial Neumann boundary measurements over $(0,T)$ of the solution to \eqref{sys:ori}, obtained by $3d+2$ times suitably changing the initial state $u_0$. 

\subsection{A short bibliography}
There is a wide mathematical literature on inverse coefficient problems for partial differential equations. Here we shall restrict our attention to references dealing with the dynamical magnetic \Schr\ equation and (we refer the reader to, e.g.,  \cite{NSU} for a global uniqueness result in an inverse problem for the magnetic static \Schr\ equation). 
In many of those works the data is the magnetic Dirichlet-to-Neumann (DN) map, which is invariant under gauge transformation of the magnetic 
potential of the \Schr\ equation. Therefore, in general, it is completely hopeless to retrieve the magnetic potential vector from the DN map. Nevertheless, the magnetic DN map does uniquely determine the magnetic field, i.e. the exterior derivative of the magnetic potential (the terminology is inherited from the three dimensional case where the exterior derivative of the magnetic potential vector is generated by its curl), see e.g. \cite{B,BC,KS}.

Notice that infinitely many boundary observations of the solution to the \Schr\ equation are needed to define the magnetic DN map. By contrast it was established in \cite{CS, HKSY} by means of a Carleman estimate that the magnetic potential of the dynamical \Schr\ equation can be stably recovered by a finite number of boundary observations of the solution over the entire time-span. The idea of using a Carleman inequality to recover unknown coefficients appearing in a partial differential equation from boundary data of the solution was first introduced by Bukhgeim and Klibanov in \cite{BK}. Since its inception in 1981 the Bukhgeim-Klibanov (BK) method was successfully applied to parabolic and hyperbolic systems, to the Maxwell equation, to the dynamical \Schr\ equation, and even to coupled systems of partial differential equations. See \cite{K} for a complete review of multidimensional inverse problems solved with this approach.

In the present article the BK method is applied to a system of two coupled magnetic \Schr\ equations. We aim for simultaneous stable determination of the two time independent electromagnetic potentials $(A^\ka,q^\ka)$, $\ka=\pm$, and the pair of coupling terms $(\phi, \Phi)$ appearing in the IBVP \eqref{sys:ori}, through finitely many Neumann data. This is reminiscent of the study carried out in \cite{KRS} where the same inverse problem is considered when $A^+=A^-=0$. Nevertheless it is worth mentioning that extending the results of \cite{KRS} to the magnetic \Schr\ system \eqref{sys:ori} is not straightforward. As a matter of fact, a specifically designed Carleman estimate borrowed from \cite{HKSY}, which is different from the one used in \cite{KRS}, is requested by the magnetic framework of this paper (see Section \ref{sec-com} for more details on this technical issue). 

Further we point out that the inverse problem of determining the linear coupling of two non-magnetic \Schr\ equations was examined in \cite{FY} and that the electric potential of two magnetic \Schr\ equations was stably retrieved by Neumann data in \cite{LT} under the assumption that $A^-=A^+$ is known. But, to the best of our knowledge, there is no reference in the mathematical literature dealing with the inverse problem of determining the electromagnetic potential of a system of two coupled \Schr\ equations, by a finite number of Neumann data.

\subsection{Notations}
Throughout the entire text, $x=(x_1,\ldots,x_d)$ denotes a generic point of $\Om \subset \R^d$. We put $\pa_i = \frac{\pa}{\pa x_i}$ for $i=1,\ldots,d$, $\pa_{ij}^2=\pa_i \pa_j$ for $i,j=1,\ldots,d$, and as usual we write $\pa_i^2$ instead of $\pa_{ii}^2$. Next, for any multi-index  $k=(k_1,k_2,\ldots,k_d) \in  \mathbb{N}_0^d$, where $\mathbb{N}_0=\{ 0 \} \cup \mathbb{N}$, we set $\pa_x^k = \pa_1^{k_1} \pa_2^{k_2} \ldots \pa_d^{k_d}$ and $\abs{k}=\sum_{j=1}^d k_j$.
Similarly, we write $\pa_t = \frac{\pa}{\pa t}$ and $\pa_\nu u = \f{\pa u}{\pa \nu} = \na u \cdot \nu$, where $\nu$ is the outward normal vector to $\Gamma$ and $\nabla$ is the gradient operator with respect to $x$. Here and below the symbol $\cdot$ stands for the Euclidian scalar product in $\R^d$ and $\nabla \cdot$ denotes the divergence operator.

Using the same the notations as in \cite{LM2} we now introduce the following functional spaces. For $X$, a manifold, we set
$$ H^{r,s}(X \times (0,T))=L^2(0,T;H^r(X)) \cap H^s(0,T;L^2(X)),\ r, s \in \left[ 0,\infty \right), $$ 
where $H^r(X)$ (resp., $H^s(0,T)$) denotes the usual Sobolev space of order $r$ (resp., $s$) in $X$ (resp., $(0,T)$), the set $H^0(X)$ (resp. $H^0(0,T)$) being understood as $L^2(X)$ (resp., $L^2(0,T)$)). More specifically, when $X=\Om$ we write
$H^{r,s}(Q)=L^2(0,T;H^r(\Om)) \cap H^s(0,T;L^2(\Om))$ instead of $H^{r,s}(\Om \times (0,T))$, while for $X=\Gamma$ we write
$H^{r,s}(\Sigma)=L^2(0,T;H^r(\Gamma)) \cap H^s(0,T;L^2(\Gamma))$ instead of $H^{r,s}(\Gamma \times (0,T))$.\\
 For further use we recall from \cite[Section 4, Theorem 2.1]{LM2} that for all $u\in H^{r,s}(X\times(0,T))$, $r,\ s>0$, and all $(j,k)\in\mathbb{N}_0^d \times\mathbb{N}_0$ such that $1-\vert j\vert/r-k/s>0$, we have
 \begin{equation}
 \label{n1}
 	\partial^j_x\partial^k_t u\in H^{\mu,\nu}(X\times(0,T))\ \textrm{where}\ \frac{\mu}{r}=\frac{\nu}{s}=1-\frac{\vert j\vert}{r}-\frac{k}{s}
 \end{equation}
 and the estimate
 \begin{equation}
 \label{n2}
 	\norm{\partial^j_x\partial^k_t u}_{H^{\mu,\nu}(X\times(0,T))}\leq\Vert u\Vert_{H^{r,s}(X\times(0,T))}.
 \end{equation}

\subsection{Main results}
We first examine the well-posedness of the IBVP \eqref{sys:ori}.
For this purpose we introduce the following Hamiltonian operator acting on $(C_0^\infty(Q)^\prime)^2$,
\begin{equation}
\label{def-HM}
\cH(A^\pm,q^\pm,\Phi,\phi)= \left( \begin{array}{cc} -\Delta_{A^+}  +q^+ & \Phi \cdot \nabla + \phi \\ -\Phi \cdot \nabla  + \phi & -\Delta_{A^-} + q^-\end{array} \right),
\end{equation}
and we rewrite the IBVP \eqref{sys:ori} as 
\begin{equation}
\label{sys:ori1}
\left\{ \begin{array}{ll} -i \partial_t u + \cH(A^\pm,q^\pm,\Phi,\phi) u = 0 & \mbox{in}\ Q \\ u(\cdot,0)=u_0 & \mbox{in}\ \Om \\ u=g & \mbox{on}\ \Sigma, \end{array} \right. 
\end{equation}
where $u=(u^+,u^-)^T$ is the transpose to $(u^+,u^-)$, $u_0=(u_0^+,u_0^-)^T$ and $g=(g^+,g^-)^T$.
Then the existence and uniqueness result for \eqref{sys:ori} that we have in mind can be stated as follows.

\begin{theorem}
\label{thm-eu}
Fix $m \in \N$ and assume that $\Gamma$ is $\mathcal{C}^{2m}$. Let $A^{\pm} \in W^{2m+1,\infty}(\Om,\R^d)$, let $q^\pm \in W^{2m,\infty}(\Om,\C)$, let $\Phi \in W^{2m,\infty}(\Om,\R^d)$ satisfy
\begin{equation}
\label{divfree}
\nabla \cdot \Phi(x)=0,\ x \in \Om,
\end{equation} 
and let
$\phi \in W^{2m,\infty}(\Om,\C)$ be such that 
$$
\norm{A^\pm}_{W^{2m+1,\infty}(\Om)^d}+ \norm{q^\pm}_{W^{2m,\infty}(\Om)}+\norm{\Phi}_{W^{2m,\infty}(\Om)^d}+ \norm{\phi}_{W^{2m,\infty}(\Om)} \leq M,
$$
for some {\it a priori} fixed positive constant $M$. \\
Then, for all
$g=(g^+,g^-)^T \in H^{2(m+\frac{3}{4}),m+\frac{3}{4}}(\Sigma)^2$ and all $u_0=(u_0^+,u_0^-)^T \in H^{2m+1}(\Om)^2$ fulfilling the following compatibility conditions
\begin{equation}
\label{d4}
\partial_t^\ell g(\cdot,0)=( -i )^\ell \left[\cH(A^\pm,q^\pm,\Phi,\phi)\right]^\ell u_0\ \mbox{on}\ \Gamma,\ \ell=0,\cdots,m-1,
\end{equation}
there exists a unique solution 
$$u=(u^+,u^-)^T \in \bigcap_{\ell=0}^{m} H^{m-\ell}(0,T;H^{2 \ell}(\Om)^2)$$ 
to the IBVP \eqref{sys:ori}. Moreover, we have
\begin{equation}
\label{d4c}
\sum_{\ell=0}^{m} \norm{u}_{H^{m-\ell}(0,T;H^{2 \ell}(\Om)^2)} \leq C \left( \norm{u_0}_{H^{2m+1}(\Om)^2}+ \norm{g}_{H^{2(m+\frac{3}{4}),m+\frac{3}{4}}(\Sigma)^2}\right),
\end{equation}
where $C$ is a positive constant depending only on $\Om$, $T$ and $M$.
\end{theorem}

It can be checked from the proof of Theorem \ref{thm-eu} displayed in Section \ref{sec-pr-eu} below, that the result of Theorem \ref{thm-eu} is still valid upon weakening the assumption $A^{\pm} \in W^{2m+1,\infty}(\Om,\R^d)$ by the following one:
$$ A^{\pm} \in W_{\div}^{2m,\infty}(\Om,\R^d)=\{ u \in W^{2m,\infty}(\Om,\R^d),\ \na \cdot u \in W^{2m,\infty}(\Om,\R) \}. $$
Nevertheless, for the sake of simplicity we stick with the above statement of Theorem \ref{thm-eu} in the remaining part of this text.

Even though the above statement looks quite similar to 
\cite[Proposition 1.1]{KRS} it is worth mentioning that Theorem \ref{thm-eu} cannot be deduced from it since $A^\pm=0$ in \cite{KRS}. Moreover we point out that the regularity and the fixed boundary values imposed on the admissible unknown coefficients by the main result of this article (see Theorem \ref{thm-main} below) are directly requested by the application of Theorem \ref{thm-eu} and namely by the compatibility conditions \eqref{d4}. Furthermore the estimate \eqref{d4c} is needed by the proof of Corollary \ref{cor1} below, which establishes that the solution to the IBVP \eqref{sys:ori} is sufficiently smooth for applying the BK method in Section \ref{sec-IP}. For all these reasons and for the convenience of the reader, we give a detailed proof of Theorem \ref{thm-eu} in Section \ref{sec-eu}.

Further, setting
\begin{equation}
\label{def-N}
N_d= \lfloor \frac{d+2}{4} \rfloor+ 3,
\end{equation}
where $\lfloor x \rfloor$ denotes the greatest integer less than or equal to $x \in \R$, we have the following useful byproduct of Theorem \ref{thm-eu}.

\begin{cor}
\label{cor1}
Under the conditions of Theorem \ref{thm-eu} with $m=N_d$, the solution $u$ to \eqref{sys:ori} lies within the class $W^{1,\infty}(0,T;W^{1,\infty}(\Om)^2)$. Moreover, there exists a positive constant $C$ depending only on $\Om$, $T$, $M$, $u_0$ and $g$, such that
\begin{equation}
\label{e00}	
\norm{u}_{W^{1,\infty}(0,T;W^{1,\infty}(\Om)^2)} \leq C.
\end{equation}
\end{cor}

This being said we turn now to stating the main result of this article, which focuses on the stability issue in the identification of the Hamiltonian $\cH(A^\pm,q^\pm,\Phi,\phi)$ by Neumann data. For this purpose we pick $M \in (0,\infty)$ and fix $m \in \{3,4,\ldots \}$. Then, given $q_0 \in W^{m,\infty}(\Om,\K)$, $\K=\R$ or $\C$, we introduce the set of admissible $\K$-valued $m$-regular scalar potentials as
\begin{align}
\label{pco}
\cQ_M^m(q_0,\K)
= & \left\{q \in W^{m,\infty}(\Om,\K):\ \norm{q}_{W^{m,\infty}(\Om)} \leq M\ \mbox{and}\ \partial_x^k q = \partial_x^k q_0^\pm\ \mbox{on}\ \Gamma,\ k=0,\ldots, m-3 \right\}.
\end{align}
Next, for $q_j \in \cQ_M^m(q_0,\K)$, $j=1,2$, we say that $(q_1,q_2) \in \digamma(\cQ_M^m(q_0,\K))$ whenever
$$ \abs{\na q_1(x)- \na q_2(x)} \le M \abs{q_1(x)-q_2(x)},\ x \in \Om. $$

Similarly, for $V_0 \in W^{m,\infty}(\Om,\R^d)$, the set of real-valued $m$-regular vector potentials is defined by
\begin{align}
\label{aco}
\cV_M^m(V_0) = & \left\{V \in W^{m,\infty}(\Om,\R^d):\ \norm{V}_{W^{m,\infty}(\Om)^d}\leq M\ \mbox{and}\ \partial_x^k V = \partial_x^k V_0\ \mbox{on}\ \Gamma,\ 0 \le \abs{k} \le m-3\ 
	 \right\}.
\end{align}
In the special case where $\nabla \cdot V_0=0$ in $\Om$, we denote by $\tilde{\cV}_M^m(V_0)$ the set of divergence free vector potentials in $\cV_M^m(V_0)$, i.e.
$$ \tilde{\cV}_M^m(V_0) = \left\{ V \in \cV_M^m(V_0):\ \na \cdot V =0\ \mbox{in}\ \Om \right\}. $$
Further, given $V_1$ and $V_2$ in $\cV_M^m(V_0)$, we write $(V_1,V_2) \in \digamma(\cV_M^m(V_0))$ if
\begin{align*} 
& \abs{\nabla ( \nabla \cdot (V_1-V_2))(x)} + \max_{i=1,\ldots,d} \sum_{j=1}^d \abs{\partial_i (V_1-V_2)_j(x)} \\
\le &  M \left(\abs{(V_1-V_2)(x)}+\abs{\nabla.(V_1-V_2)(x)}\right),\ x \in \Om, 
\end{align*}
where $(V_1-V_2)_j$, $j=1,\ldots,d$, denotes the $j$-th component of $V_1-V_2$. Likewise, for $V_1 \in \tilde{\cV}_M^m(V_0)$ and $V_2 \in \tilde{\cV}_M^m(V_0)$, we say that $(V_1,V_2) \in \digamma(\tilde{\cV}_M^m(V_0))$ when
$$ \max_{i=1,\ldots,d} \sum_{j=1}^d \abs{\partial_i (V_1-V_2)_j(x)}  \le  M \abs{(V_1-V_2)(x)},\ x \in \Om. $$
We point out that there exists
actual examples of classes of electromagnetic potentials $(A_1^\pm,A_2^\pm) \in \digamma(\tilde{\cV}_M(A_0^\pm,\R^d))$, $(q_1^\pm,q_2^\pm) \in \digamma(\cQ_M(q_0^\pm,\C))$ and coupling terms $(\Phi_1,\Phi_2)\in \digamma(\tilde{\cV}_M(\Phi_0,\R^d))$,  $(\phi_1,\phi_2)\in \digamma(\cQ_M(\phi_0,\C))$, where $A_0^\pm$, $q_0^\pm$, $\Phi_0$ and $\phi_0$ are as in Theorem \ref{thm-main}. Such examples can be built for instance by adapting the ideas of \cite[Remark, Point d)]{HKSY}. 

The main result of our article can be stated as follows.

\begin{theorem}
\label{thm-main}
Assume that $\Gamma \in \mathcal{C}^{2N_d}$ where $N_d$ is defined by \eqref{def-N}, let $A_0^\pm \in W^{2N_d+1,\infty}(\Om,\R^d)$, let $\Phi_0 \in W^{2N_d,\infty}(\Om,\R^d)$ satisfy the condition \eqref{divfree}, let $q_0^\pm \in W^{2N_d,\infty}(\Om,\R)$ and let $\phi_0 \in W^{2N_d,\infty}(\Om,\C)$. Then there exist a sub-boundary $\Gamma_0 \subset \Gamma$ and a set of $3d+2$ initial states $u_0^k=(u_0^{+,k},u_0^{-,k})^T \in H^{2N_d+1}(\Om)^2$ and boundary conditions $g^k=(g^{+,k},g^{-,k})^T \in H^{2(N_d+3 \slash 4), N_d + 3 \slash 4}(\Sigma)^2$, $k=1,\ldots, 3d+2$, fulfilling the compatibility conditions \eqref{d4} with $m=N_d$, such for all $(A_1^\pm,A_2^\pm)\in \digamma(\cV_M(A_0^\pm,\R^d))$, all $(\Phi_1,\Phi_2)\in \digamma(\tilde{\cV}_M(\Phi_0,\R^d))$, all
$(q_1^\pm,q_2^\pm) \in \digamma(\tilde{\cQ}_M(q_0,\R))$ and all $(\phi_1,\phi_2)\in \digamma(\cQ_M(\phi_0,\C))$, we have
\begin{align}
\label{se} 
&\sum_{\ka=+,-} \left( \norm{A_1^\ka-A_2^\ka}^2_{L^2(\Om)}
+\norm{\nabla.(A_1^\ka-A_2^\ka)}^2_{L^2(\Om)}
+ \norm{q_1^\ka -q_2^\ka}^2_{L^2(\Om)} \right) \\
& +\norm{\Phi_1-\Phi_2}^2_{L^2(\Om)}+\norm{\phi_1-\phi_2}^2_{L^2(\Om)} \nonumber \\
& \leq  C \sum_{k=1}^{3d+2} \sum_{\ka=+,-} \norm{\partial_\nu \partial_t u_1^{\ka,k}- \partial_\nu \partial_t u_2^{\ka,k}}^2_{L^2(\Sigma_0)}. \nonumber
\end{align}
Here, $\Sigma_0=\Gamma_0\times(0,T)$, $C$ is a positive constant depending only on $\omega$, $T$, $M$ and $(u_0^{\pm,k},g^{\pm,k})$, $k=1,\ldots,3d+2$,  and $u_j^k=(u_j^{+,k},u_j^{-,k})^T$, for $j=1,2$, is the solution to \eqref{sys:ori} given by Theorem \ref{thm-eu} where $(A_j^\pm,\Phi_j,\phi_j,q_j^\pm,u_0^{\pm,k},g^{\pm,k})$ is substituted for $(A^\pm,\Phi,\phi,q^\pm,u_0^{\pm},g^{\pm})$. 
\end{theorem}

\begin{rem}
In the special case where $\na \cdot A_0^\pm =0$ in $\Om$ and $(A_1^\pm,A_2^\pm) \in \digamma(\tilde{\cV}_M(A_0^\pm,\R^d))$, the statement of Theorem \ref{thm-main} is still valid for $(q_1^\pm,q_2^\pm) \in \digamma(\cQ_M(q_0^\pm,\C))$ and $q_0^\pm \in W^{2N_d,\infty}(\Om,\C)$, that is to say for complex-valued potentials $q_j^\pm$, $j=0,1,2$. This can be easily checked from the proof of Theorem \ref{thm-main} in Section \ref{sec-IP}.
Nevertheless, in order to avoid the inadequate expense of the size of this article we shall not elaborate on this matter.

Similarly, when the divergence of the magnetic vector potentials is known (that is to say when $\nabla \cdot  (A_1^\pm -A_2^\pm) = 0$ in $\Om$) one can see from the derivation of Theorem \ref{thm-main} that \eqref{se} remains true with only $3d$ local boundary measurements. Such a result is optimal in the sense that the $3d+3$ components of the vector-valued functions representing the unknown magnetic vector potentials, the unknown (divergence free) first order coupling vector, the unknown electric potential and the unknown zero-th order coupling coefficient, amounting altogether to $3d$ degrees\footnote{There are $3(d-1)$ degrees of freedom for the two magnetic vector potentials and the (divergence free) first order coupling vector, and $3$ more degrees for the two electric potentials and the zero-th order coupling scalar coefficient.} of freedom, are recovered with exactly $3d$ local boundary measurements of the solution. 
\end{rem}

\subsection{Comments}
\label{sec-com}
The stability inequality \eqref{se} extends the result of \cite[Theorem 1.2]{KRS} to the case of coupled magnetic \Schr\ equations with time-independent non-zero magnetic potentials, whereas in 
\cite{KRS} the $A_j^\pm$, $j=1,2$, were assumed to be zero everywhere in $\Om$. 
Although the strategy of the proof of \eqref{se} is inspired by the one of \cite[Theorem 1.2]{KRS}, there is an essential technical difference between these two approaches. As a matter of fact, in order to avoid observation at $t=0$ over the whole domain $\Om$, the authors in \cite{KRS} use a Carleman estimate on the extended domain $\Om \times (-T,T)$. As was already noticed in \cite[Section 1.2]{HKSY} and in \cite{CS}, this technique works only for \Schr\ equations with either zero or non-zero but odd time-dependent, magnetic potential. Since, here, $A_j^\pm$, $j=1,2$, are non-zero and time-independent, we will rather use the Carleman inequality stated in Theorem \ref{thm:CE} below, that was specifically designed for magnetic \Schr\ equations in \cite[Theorem 3.1]{HKSY}. This technical change is reflected in the more stringent conditions imposed by \eqref{pco}-\eqref{aco} on the unknown coefficients of the inverse problem studied in this work, relative to the ones in \cite[Theorem 1.2]{KRS}.

\subsection{Outline of the article}
The paper is organized as follows. In Section \ref{sec-eu} we establish Theorem \ref{thm-eu} and Corollary \ref{cor1}. Section \ref{sec-stb2} is devoted to the 
derivation of Theorem \ref{thm-main} while Section \ref{sec-IP} contains tools and technical results that are needed by its proof.
Details concerning the definition of the Dirichlet magnetic Laplacian can be found in the Appendix A and the proof of the relative boundedness with respect to the magnetic Laplacian, of first order perturbations is given in the Appendix B.


\section{Existence and well-posedness}
\label{sec-eu}
In this section we prove Theorem \ref{thm-eu} and Corollary \ref{cor1}. We start by defining the self-adjoint operator associated with $\cH(A^\pm,0,\Phi,0)$ for suitable real-valued vector potentials $A^\pm$ and $\Phi$.  

\subsection{Preliminaries: selfadjointness}

The technical result that we have in mind is as follows.
\begin{lem}
\label{lm-sa}
Assume that $\Ga$ is $C^2$. Let $A^\pm \in W^{1,\infty}(\Om,\R^d)$ and let $\Phi \in W^{1,\infty}(\Om,\R^d)$ fufill \eqref{divfree}.
Then, the operator 
$$H(A^\pm,\Phi) u = \begin{pmatrix}
	-\Delta_{A^+} & \Phi \cdot \na  \\
	-\Phi \cdot \na  & -\Delta_{A^-}\\
\end{pmatrix} u,\ u =(u^+,u^-)^T \in D(H(A^\pm,\Phi))=\left(H_0^1(\Om) \cap H^2(\Om)\right)^2,
$$ 
is self-adjoint in $L^2(\Om)^2$ and we have
$$ H(A^\pm,\Phi) u = \cH(A^\pm,0,\Phi,0) u,\ u \in (H_0^1(\Om) \cap H^2(\Om))^2. $$
\end{lem}
\begin{proof}
With reference to Lemma \ref{lm-mdl}, the operator
$$\Delta_{A^\pm}^D u=\begin{pmatrix}
	\Delta_{A^+} &0  \\
	0 & \Delta_{A^-}\\
\end{pmatrix} u,\ u=\left(u^+,u^-\right)^T\in D(\Delta_{A^\pm}^D)=\left(H_0^1(\Om) \cap H^2(\Om)\right)^2,$$
is self-adjoint in $L^2(\Om)^2$.
Next, since $\Phi \in W^{1,\infty}(\Om,\R^d)$ is divergence free, the operator
$$
T_\Phi u =\begin{pmatrix}
	0&\Phi \cdot \nabla\\
	-\Phi \cdot \nabla&0\\
\end{pmatrix} u,\ u=\left(u^+,u^-\right)^T \in D(T_\Phi)=\left(H_0^1(\Om) \right)^2,
$$
is symmetric in $L^2(\Om)^2$. Moreover, $T_\Phi$ is $\Delta_{A^\pm}^D$-bounded, with relative bound zero, according to Lemma \ref{lm-rbp}. Therefore, by Kato-Rellich Theorem (see \cite[Theorem X.12]{MRBS}), the operator $-\Delta_{A^\pm}^D+T_\Phi$, endowed with the domain $\left(H_0^1(\Om) \cap H^2(\Om)\right)^2$ is self-adjoint in $L^2(\Om)^2$ . 
\end{proof}

Next, we introduce for further use the following multiplication operator in $L^2(\Om)^2$ by the functions $q^\pm \in W^{2,\infty}(\Om,\C)$ and $\phi \in W^{2,\infty}(\Om,\C)$, as
$$P_{q^\pm,\phi}=\begin{pmatrix}
	q^+& \phi \\
	\phi & q^-\\
\end{pmatrix}.
$$
Since $D(H(A^\pm,\Phi)))=(H_0^1(\Omega) \cap H^2(\Omega))^2$ and since $q^\pm$ and $\phi$ are taken in $W^{2,\infty}(\Omega)$, it is apparent that \begin{equation}
\label{e01}
P_{q^\pm,\phi} \in C([0,T],\cB(D(H(A^\pm,\Phi)))),
\end{equation} 
where $\cB(X)$ denotes the space of linear bounded operators in the Banach space $X$.

\subsection{Proof of Theorem \ref{thm-eu}}
\label{sec-pr-eu}
We proceed by induction on $m$. \\

\noindent {\it Base step}. We start by proving the statement of Theorem \ref{thm-eu} for $m=1$.
Since $g=(g^+,g^-)^T \in H^{\frac{7}{2},\frac{7}{4}}(\Sigma)^2$ and $u_0=(u_0^+,u_0^-)^T \in H^3(\Om)^3$, we apply \cite[Chapter 4, Theorem 2.1]{LM2} and get $G=(G^+,G^-)^T\in H^{4,2}(Q)^2$ satisfying $G=g$ on $\Sigma$, $G(\cdot,0)=u_0$ in $\Om$, and the estimate
\begin{equation}
\label{d1}
\norm{G}_{H^{4,2}(Q)^2} \leq C \left( \norm{u_0}_{H^3(\Om)^2} + \norm{g}_{H^{\frac{7}{2},\frac{7}{4}}(\Sigma)^2} \right).
\end{equation}
Here and in the remaining part of this proof, $C$ denotes a generic positive constant depending only on $\Om$, $T$ and $M$, which may change from line to line.

Next we notice that $u=(u^+,u^-)^T$ is a solution to \eqref{sys:ori} if and only if $v=u-G=(u^+-G^+,u^- - G^-)^T$ solves
\begin{equation}
\label{d2}		
\left\{
\begin{array}{ll}
-i \partial_t v + \left( H(A^\pm,\Phi)  + M_{q^\pm,p} \right) v=f &\textrm{in}\ Q \\
v(\cdot,0)=0 &\textrm{in}\ \Om,
\end{array}
\right.
\end{equation}
where $f=(f^+,f^-)^T=i \partial_t G - \cH(A^\pm,q^\pm,\Phi,\phi) G$.
In light of \eqref{n1}-\eqref{n2}, the functions $\partial_t G$ and $\cH(A^\pm,q^\pm,\Phi,\phi) G$ are both in $H^{2,1}(Q)^2$ and they satisfy the estimate
$$\norm{\partial_t G}_{H^{2,1}(Q)^2} + \norm{\cH(A^\pm,q^\pm,\Phi,\phi) G}_{H^{2,1}(Q)^2}\leq C \norm{G}_{H^{4,2}(Q)^2}. $$
Thus we have $f \in H^{0,1}(Q)=H^1(0,T;L^2(\Om)^2)$ and
\begin{align}
\label{d3}
\norm{f}_{H^1(0,T;L^2(\Om)^2)} & \leq C \left( \norm{\partial_t G}_{H^{2,1}(Q)^2} +\norm{\cH(A^\pm,q^\pm,\Phi,\phi) G}_{H^{2,1}(Q)^2} \right) \\
& \leq C \norm{G}_{H^{4,2}(Q)^2}. \nonumber 
\end{align}
Moreover, due to \eqref{e01} and the maximal dissipativity of the operator $-iH(A^\pm,\Phi)$ which follows readily from Lemma \ref{lm-sa}, we may apply \cite[Lemma 2.1]{CKS} with $X=L^2(\Om)^2$, $M_0=-iH(A^\pm,\Phi)$ and $B=-i P_{q^\pm,\phi}$. We obtain that \eqref{d2} admits a unique solution $v \in H^{2,1}(Q)^2$ such that
$$
\norm{v}_{H^{2,1}(Q)^2} \leq C \norm{f}_{H^1(0,T;L^2(\Om)^2)}.
$$
Finally, using that $u=v+G$, we get \eqref{d4c} by combining the above estimate with \eqref{d1} and \eqref{d3}.\\

 \noindent {\it Inductive step}. Fix $m \in \{ 2, 3, \ldots \}$ and assume that the statement of Theorem \ref{thm-eu} where $m-1$ is substituted for $m$, holds. 
Let $u$ be the $\bigcap_{\ell=0}^{m-1} H^{m-1-\ell}(0,T;H^{2 \ell}(\Om)^2)$-solution to \eqref{sys:ori1} obtained by applying Theorem \ref{thm-eu} where $m$ was replaced by $m-1$, satisfying
\begin{equation}
\label{f30}
\sum_{\ell=0}^{m-1} \norm{u}_{H^{m-1-\ell}(0,T;H^{2 \ell}(\Om)^2)} \leq C \left( \norm{u_0}_{H^{2m-1}(\Om)^2}+ \norm{g}_{H^{2(m-\frac{1}{4}),m-\frac{1}{4}}(\Sigma)^2}\right),
\end{equation}
according to \eqref{d4c}. In particular we have $u \in H^1(0,T;H^{2(m-2)}(\Om)^2)$ and the function $w=\partial_t u$ solves
\begin{equation}
\label{f31}		
\left\{
\begin{array}{ll}
-i \partial_t w + \cH(A^\pm,q^\pm,\Phi,\phi) w = 0 &\mbox{in}\ Q \\
w = \partial_t g & \mbox{in}\ \Sigma \\
w(\cdot,0)=w_0 &\mbox{in}\ \Om,
\end{array}
\right.
\end{equation}
where $w_0=-i \cH(A^\pm,q^\pm,\Phi,\phi) u_0 \in H^{2m-1}(\Om)^2$
satisfies $\norm{w_0}_{H^{2m-1}(\Om)^2} \le C \norm{u_0}_{H^{2m+1}(\Om)^2}$ and
$\partial_t g \in H^{2 \left( m-\frac{1}{4}\right),m-\frac{1}{4}}(\Sigma)^2$ fulfills
$\norm{\partial_t g}_{H^{2 \left( m-\frac{1}{4}\right),m-\frac{1}{4}}(\Sigma)^2} \le \norm{g}_{H^{2 \left( m+\frac{3}{4}\right),m+\frac{3}{4}}(\Sigma)^2}$, from \eqref{n1}-\eqref{n2}. Moreover, for all $\ell=0,1,\ldots,m-2$, we have
\begin{align*}
\partial_t^\ell (\partial_t g)(\cdot,0)  & = \partial_t^{\ell+1} g(\cdot,0) \\
& = (-i)^{\ell+1} \left[ \cH(A^\pm,q^\pm,\Phi,\phi) \right]^{\ell+1} u_0
\end{align*}
according to \eqref{d4c},
whence 
$$\partial_t^\ell (\partial_t g)(\cdot,0)=(-i)^\ell \left[ \cH(A^\pm,q^\pm,\Phi,\phi)\right]^{\ell} w_0. $$
Therefore we have $w \in \bigcap_{\ell=0}^{m-1} H^{m-1-\ell}(0,T;H^{2 \ell}(\Om)^2)$ according to the induction hypothesis, and the estimate
\begin{align}
\label{f31b}
\sum_{\ell=0}^{m-1} \norm{w}_{H^{m-1-\ell}(0,T;H^{2 \ell}(\Om)^2)} & \le C \left( \norm{w_0}_{H^{2m-1}(\Om)^2} + \norm{\partial_t g}_{H^{2 \left( m - \frac{1}{4} \right) , m - \frac{1}{4}}(\Sigma)^2} \right) \\
& \le C \left( \norm{u_0}_{H^{2m+1}(\Om)^2} + \norm{g}_{H^{2 \left( m+\frac{3}{4}\right),m+\frac{3}{4}}(\Sigma)^2}\right). \nonumber 
\end{align}
From this and \eqref{f30} it then follows that $u \in \bigcap_{\ell=0}^{m-1} H^{m-\ell}(0,T;H^{2 \ell}(\Om)^2)$ verifies
\begin{equation}
\label{f32}
\sum_{\ell=0}^{m-1} \norm{u}_{H^{m-\ell}(0,T;H^{2 \ell}(\Om)^2)}  \le C \left( \norm{u_0}_{H^{2m+1}(\Om)^2} + \norm{g}_{H^{2 \left( m+\frac{3}{4}\right),m+\frac{3}{4}}(\Sigma)^2}\right).
\end{equation}
Thus it remains to show that $u \in L^2(0,T;H^{2m}(\Om)^2)$ and that 
$$ \norm{u}_{L^2(0,T;H^{2 m}(\Om)^2)}  \le C \left( \norm{u_0}_{H^{2m+1}(\Om)^2} + \norm{g}_{H^{2 \left( m+\frac{3}{4}\right),m+\frac{3}{4}}(\Sigma)^2}\right). $$

To do that we notice from \eqref{def-deltaA} and \eqref{def-HM}-\eqref{sys:ori1} that for a.e. $t \in (0,T)$, the function $u(\cdot,t)$ is a solution to the following boundary value problem (BVP)
\begin{equation}
\label{f21b}
\left\{
\begin{array}{ll}
-\Delta u(\cdot,t) = \psi(\cdot,t) &\mbox{in}\ Q \\
u(\cdot,t) = g(\cdot,t) & \mbox{on}\ \Gamma,
\end{array}
\right.
\end{equation}
where $\De u=(\De u^+,\De u^-)^T$ and $\psi = (\psi^+,\psi^-)^T$ is expressed by
\begin{align}
\label{f22}
\psi^{\pm}(x,t)  = & \left( i \partial_t + 2 i A^\pm(x) \cdot \na  + i (\na \cdot A^\pm(x)) -\abs{A^\pm(x)}^2  -q^\pm(x) \right) u^{\pm}(x,t) \\
& + \left(  \mp \Phi(x) \cdot \na - \phi(x) \right) u^\mp(x,t),\ x \in \Om,\ t \in (0,T). \nonumber
\end{align}
Using the elliptic regularity property of the BVP \eqref{f21b} we may now improve the spatial regularity of $u$ as follows.
Namely, since $u \in H^1(0,T;H^{2(m-1)}(\Om)^2)$, it follows from \eqref{f22} that $\psi(\cdot,t) \in H^{2m-3}(\Om)^2$ for a.e. $t \in (0,T)$. Moreover, since $g(\cdot,t) \in H^{2 \left( m - \frac{3}{4}\right)}(\Gamma)^2$ and $\Gamma \in C^{2m}$, we deduce from \eqref{f21b} that $u(\cdot,t) \in H^{2m-1}(\Om)^2$ satisfies
\begin{align*}
\norm{u(\cdot,t)}_{H^{2m-1}(\Om)^2} & \le C \left( \norm{\psi(\cdot,t)}_{H^{2m-3}(\Om)^2} + \norm{g(\cdot,t)}_{H^{2 \left( m - \frac{3}{4} \right)}(\Gamma)^2} \right) \\
& \le C \left(  \norm{w(\cdot,t)}_{H^{2m-3}(\Om)^2} +  \norm{u(\cdot,t)}_{H^{2(m-1)}(\Om)^2} + \norm{g(\cdot,t)}_{H^{2 \left( m - \frac{3}{4} \right)}(\Gamma)^2} \right). \nonumber 
\end{align*}
This entails that $u \in L^2(0,T;H^{2m-1}(\Om)^2)$ and that
\begin{align}
\label{f23}
\norm{u}_{L^2(0,T;H^{2m-1}(\Om)^2)} & \le C \left(  \norm{w}_{L^2(0,T;H^{2m-3}(\Om)^2)} +  \norm{u}_{L^2(0,T;H^{2(m-1)}(\Om)^2)} + \norm{g}_{L^2(0,T;H^{2 \left( m - \frac{3}{4} \right)}(\Gamma)^2)} \right) \\
& \le C \left(  \norm{w}_{L^2(0,T;H^{2(m-1)}(\Om)^2)} +  \norm{u}_{L^2(0,T;H^{2(m-1)}(\Om)^2)} + \norm{g}_{L^2(0,T;H^{2 \left( m - \frac{3}{4} \right)}(\Gamma)^2)} \right). \nonumber 
\end{align}
Moreover we have
$\psi(\cdot,t) \in H^{2(m-1)}(\Om)^2$ for a.e. $t \in (0,T)$, by \eqref{f22}, and since $g(\cdot,t) \in H^{2 \left( m-\frac{1}{4} \right)}(\Gamma)^2$ and $\Gamma \in C^{2m}$, we infer from the BVP \eqref{f21b} that $u(\cdot,t) \in H^{2m}(\Om)^2$ and that
\begin{align*}
\norm{u(\cdot,t)}_{H^{2m}(\Om)^2} & \le C \left( \norm{\psi(\cdot,t)}_{H^{2(m-1)}\Om)^2} + \norm{g(\cdot,t)}_{H^{2 \left( m-\frac{1}{4} \right)}(\Gamma)^2} \right) \\
& \le C \left(  \norm{w(\cdot,t)}_{H^{2(m-1)}(\Om)^2} +  \norm{u(\cdot,t)}_{H^{2m-1}(\Om)^2} + \norm{g(\cdot,t)}_{H^{2 \left( m-\frac{1}{4} \right)}(\Gamma)^2} \right). \nonumber 
\end{align*}
From this and \eqref{f23} it then follows that $u \in L^2(0,T;H^{2m}(\Om)^2)$ satisfies
\begin{align*}
\norm{u}_{L^2(0,T;H^{2m}(\Om)^2)} & \le C \left(  \norm{w}_{L^2(0,T;H^{2(m-1)}(\Om)^2)} +  \norm{u}_{L^2(0,T;H^{2m-1}(\Om)^2)} + \norm{g}_{L^2(0,T;H^{2 \left( m-\frac{1}{4} \right)}(\Gamma)^2)} \right) \\
& \le C \left( \norm{w}_{L^2(0,T;H^{2(m-1)}(\Om)^2)} +  \norm{u}_{L^2(0,T;H^{2(m-1)}(\Om)^2)} + \norm{g}_{L^2(0,T;H^{2 \left( m-\frac{1}{4} \right)}(\Gamma)^2)} \right),
\end{align*}
where we used in the last line that $\norm{u}_{L^2(0,T;H^{2m-1}(\Om)^2)} \le \norm{w}_{L^2(0,T;H^{2(m-1)}(\Om)^2)} +  \norm{u}_{L^2(0,T;H^{2(m-1)}(\Om)^2)}$.
Putting this together with \eqref{f31b}-\eqref{f32} we end up getting
 \eqref{d4c}, which terminates the proof. 
 
\subsection{Proof of Corollary \ref{cor1}}
An application of Theorem \ref{thm-eu} with $m=N_d$ shows that $u \in H^2(0,T;H^{2(N_d-2)}(\Om)^2)$. Further, since
$2(N_d-2) > d/2+1$ from \eqref{def-N}, we have 
$$u \in W^{1,\infty}(0,T;W^{1,\infty}(\Om)^2)$$ 
by Sobolev embedding theorem, and the estimate
$$ \norm{u}_{W^{1,\infty}(0,T;W^{1,\infty}(\Om)^2)}  \leq  c \norm{u}_{H^2(0,T;H^{2(N_d-2)}(\Om)^2)}, $$
where $c$ is a positive constant depending only on $\Om$, $T$ and $M$. Finally, \eqref{e00} follows from this and \eqref{d4c}.

\section{Inverse problem: tools and preliminaries}
\label{sec-IP}
The analysis of the inverse problem studied in this work is built upon the ideas of the BK method, which mostly 
relies on a global Carleman estimate for the operator
\begin{equation}
	\label{equ-gov}
	L=-i \pa_t  - \De  
\end{equation}
acting in $\left( C_0^\infty(Q) \right)^{\prime}$. This inequality is presented in the coming section.

\subsection{A Carleman inequality}
\label{sec-CE}
Let us consider a nonnegative function $\title{\beta}\in \mathcal{C}^4(\overline{\Om})$ satisfying the two following conditions:
\begin{enumerate}[(i)]
\item $\beta$ has no critical point in $\Om$: 
\begin{align}
\label{beta1}
\exists c_0 \in (0,\infty),\ \abs{\na\be(x)} \ge c_0,\ x \in \Om;
\end{align}
\item $\beta$ is pseudo-convex with respect to the Laplace operator: 
\begin{align}
\label{beta2}
\exists (\lambda_0,\ep) \in(0,\infty)^2,\ \la \ge \la_0 \Longrightarrow \la \abs{\na \be(x) \cdot \xi}^2 + \sum_{i,j=1}^d \pa_{ij}^2 \be(x) \xi_i \xi_j \ge \ep \abs{\xi}^2,\ x \in \Om,\ \xi\in \Bbb{R}^d.
\end{align} 
\end{enumerate}
Further we pick a function $\ell \in C^1([0,T],[0,\infty))$ obeying
\begin{equation}
	\label{con:l}
\quad\ell(T) = 0\ \mbox{and}\ 0 \le \ell(t)<\ell(0),\ t\in (0,T].
\end{equation}
There are numerous examples of functions $\beta$ and $\ell$ fulfilling the conditions \eqref{beta1}-\eqref{beta2} and \eqref{con:l} respectively. For instance, 
$\beta(x)=\abs{x-x_0}^2$ for all $x \in \overline{\Om}$ where $x_0$ is an arbitrary fixed point in $\R^d \setminus \overline{\Om}$ and $\ell(t)=(T-t)(T+t)$ for all $t \in [0,T]$, are one of them and it is well known in this peculiar case that the observation zone of the Neumann data used in the analysis of the inverse problem studied in this work, defined by
\begin{equation}
	\label{def-Ga0}
	\Ga_0 = \{x\in\pa\Om:  \na \be(x) \cdot \nu(x) \ge 0\},
\end{equation}
is the $x_0$-shadowed face of the boundary $\Ga$, see e.g. \cite{BP02,HKSY}). 

Next we introduce the following weight function  
\begin{align}
\label{con:weight}
	\al(x,t) = \f{e^{2\la\be(x)} -e^{\la K}}{\ell^2(t)},\ (x,t) \in Q, 
\end{align}
where $K=2\sup_{x\in\Om}\be(x)$ and $\lambda \in [1,\infty)$ is taken so large relative to $K$ that $\lambda \ge 2(\ln 2)K-1$. 
This being said 
we notice from \eqref{equ-gov} through direct calculation that we have
\begin{align*}
	e^{s\al}Le^{-s \al} 
	= is(\pa_t \al) + R_1+R_2,\ s \in (0,\infty),
\end{align*}
where 
\begin{equation}
\label{def:R}
R_1 = -i\pa_t  - \De  - s^2|\na\al|^2\ \mbox{and}\ R_2  = 2s\na\al\!\cdot\!\na  + s(\De\al). 
\end{equation}


%
Then we have the following global Carleman estimate for the operator \eqref{equ-gov} as a straightforward consequence of \cite[Theorem 3.1 and Remark 3.2]{HKSY}.
\begin{thm}
\label{thm:CE}
Let $L$, $\be$, $\ell$, $\la$, $\al$ and $R_j$ for $j=1,2$, be as above. Then, there exist two positive constants $s_0$ and $C$, both of them depending only on $\ep$, $c_0$, $\|\be\|_{L^\infty(\Om)}$, $\ell(0)$, $\|\ell^\prime\|_{L^\infty(\Om)}$ and $\la$, such that for all $s \in (s_0,\infty)$ and all $u\in L^2(0,T; H_0^1(\Om))$ satisfying $Lu\in L^2(Q)$ and $\pa_\nu u\in L^2(\Sigma)$, we have	
\begin{align}
\label{CE}
& \|R_1 (e^{s\al} u)\|_{L^2(Q)}^2 + \|R_2 (e^{s\al} u)\|_{L^2(Q)}^2 + s\|e^{s\al}\na u\|_{L^2(Q)}^2 + s^3\|e^{s\al} u\|_{L^2(Q)}^2
		 \\
\le & C \left( \|e^{s\al}Lu\|_{L^2(Q)}^2 +   \left\|    \pa_\nu u \right\|_{L^2(\Sigma_0)}^2 + s \cI(u(\cdot,0))\right), \nonumber
\end{align}
where $\Sigma_0=\Ga_0 \times (0,T)$ and
\begin{equation}
	\label{defI}
	\cI(u(\cdot,0)) = \int_{\Om} e^{2s\al(x,0)} \abs{\na\be(x) \cdot ( \ov{u} \na u- u \na \ov{u})(x,0)} dx.
\end{equation} 
\end{thm} 

The first step of the BK method is to linearize the system under study. 


\subsection{Linearized system}
We mean by this that we write $u_j=(u^+_j,u^-_j)^T$, $j=1,2$, for the solution to \eqref{sys:ori} where $(A_j^\pm,q_j^\pm,\Phi_j^\pm,\phi_j^\pm)$ is substituted for $(A^\pm,q^\pm,\Phi,\phi)$ and we take the difference of the two corresponding systems. This way, putting $A^\pm=A_1^\pm-A_2^\pm$, $q^\pm=q_1^\pm-q_2^\pm$, $\Phi=\Phi_1-\Phi_2$ and $\phi=\phi_1-\phi_2$, we obain that $u=(u^+,u^-)^T= (u_1^+-u_2^+,u_1^- - u_2^-)^T$ solves
\begin{equation}
\label{p1}
\left\{
\begin{array}{ll}
-i\partial_t u + \cH(A_1^\pm,q_1^\pm,\Phi_1,\phi_1) u =  \cG(A^\pm,q^\pm,\Phi,\phi) u_2&\textrm{in}\ Q \\
		u(\cdot,0)=0&\textrm{in}\ \Om\\
		u=0&\textrm{on}\ \Sigma,
	\end{array}
	\right.
\end{equation}
where
$$
\cG(A^\pm,q^\pm,\Phi,\phi)= \begin{pmatrix} \vartheta_{\cS^+}(A^+, q^+) & - \Theta(\Phi,\phi) \\ -\Theta(-\Phi,\phi) & \vartheta_{\cS^-}(A^-, q^-) \end{pmatrix}, 
$$
the notation $\mathcal{S}^\pm$ being a shorthand for $A_1^\pm + A_2^\pm$, and
\begin{equation}
\label{def-theta}
\vartheta_{\cS^\pm}(A^\pm, q^\pm)=2 i A^\pm \cdot \nabla +i (\na_{\cS^\pm} \cdot A^\pm)  - q^\pm,\ \Theta(\pm\Phi,\phi)=\pm\Phi \cdot \nabla + \phi.
\end{equation}
Here and in the remaining part of this text, $\na_X$ for $X \in \C^d$, denotes the $X$-magnetic gradient operator 
$$\na_X= \nabla + i X. $$
Since $u^\pm \in H^2(0,T;L^2(\Om)) \cap H^1(0,T;H^2(\Om)\cap H^1_0(\Om))$ we may differentiate \eqref{p1} with respect to the time-variable. This is the second step of the BK method and as we shall see below its main benefit is that the initial state of the new system is expressed as a function of the unknown parameters $A^\pm$, $q^\pm$, $\Phi$ and $\phi$.

\subsection{Time-differentiation} 
Setting 
$v^\pm=\partial_t u^\pm \in H^1(0,T;L^2(\Om))\cap L^2(0,T;H^2(\Om)\cap H^1_0(\Om))$, we get that $v=\partial_t u=(v^+,v^-)^T$ is the solution to the following coupled system
$$
\left\{
\begin{array}{ll}
-i\partial_t v + \cH(A_1^\pm,q_1^\pm,\Phi_1,\phi_1) v =  \cG(A^\pm,q^\pm,\Phi,\phi) \partial_t u_2 &\mbox{in}\ Q \\
v(\cdot,0)=  i \cG(A^\pm,q^\pm,\Phi,\phi) u_0 &\textrm{in}\ \Om\\
v=0&\mbox{on}\ \Sigma. \\
\end{array}
\right.
$$
In light of \eqref{equ-gov} this can be equivalently rewritten as
\begin{equation}
\label{p3}
\left\{
\begin{array}{ll}
Lv^\pm = f^\pm &\mbox{in}\ Q \\
v^\pm(\cdot,0)=v_0^\pm
&\mbox{in}\ \Om\\
v^\pm=0 &\mbox{on}\ \Sigma,
\end{array}
\right.
\end{equation}
where
\begin{equation}
\label{def-v0pm}
v_0^\pm=i \left( \vartheta_{\cS^\pm}(A^\pm,q^\pm) u_0^\pm - \Theta(\pm \Phi,\phi) u_0^\mp \right)
\end{equation}
and
\begin{equation}
\label{def-fpm}
f^\pm = \vartheta_{A_1^\pm}(A_1^\pm, q_1^\pm)v^\pm - \Theta( \mp \Phi_1,\phi_1) v^\mp + \vartheta_{\cS^\pm}(A^\pm, q^\pm)\partial_t u_2^\pm - \Theta( \mp \Phi,\phi) \partial_t u_2^\mp.
\end{equation}
As expected, $v_0^\pm$ is expressed in terms of the unknowns $A^\pm$, $q^\pm$, $\Phi$ and $\phi$. 
In light of the right hand side of \eqref{def-v0pm} we are thus left with the task 
of extracting the relevant information on each of these unknown coefficients by suitably choosing the initial states $u_0^\pm$. This will be carried out below
in Section \ref{sec-eop} but prior to doing that, we shall estimate the $e^{s \al(\cdot,0)}$-weighted $L^2(\Om)$-norm of $v_0^\pm$, $s \in (0,\infty)$, with the following lemma, whose proof can be found in \cite[Section 4.1]{HKSY}.

\begin{lem}
For all $s \in (0,\infty)$, we have
\begin{equation}
\label{eq:lem}
\norm{e^{s \al_0} v^\pm (\cdot,0)}_{L^2(\Om)}^2 \le s^{-\f{3}{2}} \left(\norm{R_1 e^{s\al} v^\pm}_{L^2(Q)}^2 + s^3 \norm{e^{s \al} v^\pm}_{L^2(Q)}^2 \right),
\end{equation}
where $\alpha_0(x)=\alpha(x,0)$ for all $x \in \overline{\Om}$ and $R_1$ is defined by \eqref{def:R}.
\end{lem}

We shall see in Section \ref{sec-obs} that the right hand side of \eqref{eq:lem} can be bounded (up to some multiplicative constant) by the Neumann data of the inverse problem under scrutiny, upon taking $s$ sufficiently large and applying the Carleman estimate of Theorem \ref{thm:CE}.


\section{Proof of Theorem \ref{thm-main}}
\label{sec-stb2}

We start by showing that the unknowns coefficients $A^\pm$, $q^\pm$, $\Phi$ and $\phi$ can be observed by the Neumann data associated with the IBVP \eqref{p3}.

\subsection{Observation inequality}
\label{sec-obs}
Since $v^\pm \in L^2(0,T;H_0^1(\Om))$, we have $L v^\pm \in L^2(0,T;L^2(\Om))$ and $\pa_\nu v\in L^2(0,T; L^2(\pa\Om))$, hence we can apply the Carleman estimate of Theorem \ref{thm:CE} to $v^\pm$. We get that
\begin{align}
\label{aqzs}
& \sum_{\ka=+,-} \left( \sum_{j=1}^2 \|R_j (e^{s\al} v^\ka)\|_{L^2(Q)}^2 + s\|e^{s\al}\na v^\ka\|_{L^2(Q)}^2 + s^3\|e^{s\al} v^\ka\|_{L^2(Q)}^2 \right)  \\
\le & C \sum_{\ka=+,-} \left(  \|e^{s\al}Lv^\ka\|_{L^2(Q)}^2
+\left\|   \pa_\nu v^\ka \right\|_{L^2(\Sigma_0)}^2  + s  \cI(v_0^\ka) \right),\ s \in (s_0,\infty).
\nonumber
\end{align}
Next, with reference to the first line of \eqref{p3} and to \eqref{def-fpm} we have
\begin{equation}
\label{e7}
\norm{e^{s\al}Lv^\pm}_{L^2(Q)} \le V_{A_1^\pm}(A_1^\pm,q_1^\pm,v^\pm) + T(\Phi_1,\phi_1,v^\mp) +
V_{\cS^\pm}(A^\pm, q^\pm,\partial_t u_2^\pm) + T(\Phi,\phi,\partial_t u_2^\mp),
\end{equation}
where
\begin{equation} 
\label{def-V}
V_{X}(Y,q,w)=2 \norm{e^{s\al} Y \cdot \nabla w}_{L^2(Q)} +\norm{e^{s\al} (\nabla \cdot Y) w}_{L^2(Q)}
+\norm{ e^{s\al} (X \cdot Y) w}_{L^2(Q)} +\norm{e^{s\al}  q w}_{L^2(Q)}
\end{equation}
and
\begin{equation}
\label{def-T}
T(\Phi,\phi,w)=\norm{e^{s\al} \Phi \cdot \nabla w}_{L^2(Q)}
+\norm{e^{s\al} \phi w}_{L^2(Q)}.
\end{equation}
Further, using that 
$\norm{\Phi_1}_{L^\infty(\Om)^d} \le M$, $\norm{\phi_1}_{L^\infty(\Om)} \le M$, $\norm{A_1^\pm}_{W^{1,infty}(\Om)^d} \le M$ and $\norm{q_1^\pm}_{L^\infty(\Om)} \le M$ (as we have $\Phi_1 \in \cV_M(\Phi_0)$, $\phi_1 \in\cQ_M(\phi_0)$, $A_1^ \pm \in \cV_M(A_0^\pm)$ and $q_1^ \pm \in\cQ_M(q_0^\pm)$ by assumption), we obtain that for all $s \in (0,\infty)$ ,
\begin{align}
\label{e8}
& V_{A_1^\pm}(A_1^\pm,q_1^\pm,v^\pm) + T(\Phi_1,\phi_1,v^\mp) \\
\le & M \left(  2\norm{e^{s\al} \nabla v^\pm}_{L^2(Q)^d} + (2+M) \norm{e^{s\al}v^\pm}_{L^2(Q)} + \norm{e^{s\al} \nabla v^\mp}_{L^2(Q)^d} + \norm{e^{s\al}v^\mp}_{L^2(Q)} \right). \nonumber
\end{align}
Similarly, using that $\norm{\cS^\pm \cdot A^\pm}_{L^\infty(\Om)} \le 2M \norm{A^\pm}_{L^\infty(\Om)^d}$ since $A_j^\pm \in \cV_M(A_0^\pm)$, $j=1,2$, and that $\norm{\partial_t u_2^\pm}_{L^\infty(Q)}$ and $\norm{\nabla\partial_t u_2^\pm}_{L^\infty(Q)^d}$ are bounded in accordance with \eqref{e00}, we get that
\begin{align*}
& V_{\cS^\pm}(A^\pm, q^\pm,\partial_t u_2^\pm) + T(\Phi,\phi,\partial_t u_2^\mp) \\
\le & C \left( 2 \norm{e^{s\al} A^\pm}_{L^2(Q)^d} + (1+2M) \norm{e^{s\al} \nabla \cdot A^\pm}_{L^2(Q)} + \norm{e^{s\al} q^\pm}_{L^2(Q)} +\norm{e^{s\al} \Phi}_{L^2(Q)^d}
	+\norm{e^{s\al} \phi}_{L^2(Q)}\right).
\end{align*}
Putting this together with \eqref{e7}-\eqref{e8} and using that $\alpha(x,t) \le \alpha_0(x)$ for all $x \in \Om$ and all $t \in (0,T)$, we find that
\begin{align}
\label{e9}
& \norm{e^{s\al} Lv^\pm}_{L^2(Q)}^2 \\
\le & C \left( \sum_{\ka=+,-} \left( \norm{e^{s\al} v^\ka}_{L^2(Q)}^2 + 
	\norm{e^{s\al} \nabla v^\ka}_{L^2(Q)^d}^2 \right) + \mathfrak{h}_s(A^\pm,q^\pm,\Phi,\phi) \right),\ s \in (0,\infty), \nonumber
\end{align}
where
\begin{align}
\label{e10}
& \mathfrak{h}_s(A^\pm,q^\pm,\Phi,\phi) \\
= & \sum_{\ka=+,-} \left( \norm{e^{s\al_0} A^\ka}_{L^2(\Om)^d}^2+ \norm{e^{s\al_0} \nabla \cdot A^\ka}_{L^2(\Om)}^2 + \norm{e^{s\al_0} q^\ka}_{L^2(\Om)}^2 \right) + \norm{e^{s\al_0} \Phi}_{L^2(\Om)^d}^2+ \norm{e^{s\al_0} \phi}_{L^2(\Om)}^2. \nonumber
\end{align}
Next, inserting \eqref{e9} into \eqref{aqzs} we obtain
\begin{align*}
& \sum_{\ka=+,-} \left( \sum_{j=1}^2 \| R_j (e^{s\al} v^\ka) \|_{L^2(Q)}^2 + (s-2C) \|e^{s\al} \na v^\ka \|_{L^2(Q)^d}^2 + (s^3-2C) \|e^{s\al} v^\ka \|_{L^2(Q)}^2  \right) \\
\le & C \left( \sum_{\ka=+,-} \left( \norm{\pa_\nu v^\ka}_{L^2(\Sigma_0)}^2  + s  \cI(v_0^\ka) \right) + \mathfrak{h}_s(A^\pm,q^\pm,\Phi,\phi) \right),\ s \in (s_0,\infty).
\end{align*}
Taking $s_1\in (s_0,\infty)$ so large that, say, $\min(s_1,s_1^3)\geq 3C$, we infer from the above estimate that
\begin{align*}
& \sum_{\ka=+,-} \left( \sum_{j=1}^2 \norm{R_j (e^{s\al} v^\ka)}_{L^2(Q)}^2 + s \norm{e^{s\al}\na v^\ka}_{L^2(Q)^d}^2 + s^3 \norm{e^{s\al} v^\ka}_{L^2(Q)}^2  \right) \\
\le & C \left( \sum_{\ka=+,-} \left( \norm{\pa_\nu v^\ka}_{L^2(\Sigma_0)}^2  + s  \cI(v_0^\ka) \right) + \mathfrak{h}_s(A^\pm,q^\pm,\Phi,\phi) \right),\ s \in (s_1,\infty).
\end{align*}
This and \eqref{eq:lem} then yield that
$$
\sum_{\ka=\pm} \norm{e^{s \al_0} v_0^\ka}_{L^2(\Om)}^2 
\le  C s^{-\frac{3}{2}} \left( \sum_{\ka=+,-} \left( \norm{\pa_\nu v^\ka}_{L^2(\Sigma_0)}^2  + s  \cI(v_0^\ka) \right) + \mathfrak{h}_s(A^\pm,q^\pm,\Phi,\phi) \right),\ s \in (s_1,\infty).
$$
The second term on the right hand side of the above inequality is treated by the following technical estimate, whose proof is postponed to Section \ref{sec-prlm0}.
\begin{lem}
\label{lm0}
For all $s \in (s_1,\infty)$, we have
$$  \sum_{\ka=+,-} \cI(v_0^\ka) \le C \mathfrak{h}_s(A^\pm,q^\pm,\Phi,\phi). $$
\end{lem}
In light of this, we obtain that
\begin{equation}
\label{e11}
\sum_{\ka=\pm} \norm{e^{s \al_0} v_0^\ka}_{L^2(\Om)}^2
\le C s^{-\frac{3}{2}} \left( \sum_{\ka=+,-} \norm{\pa_\nu v^\ka}_{L^2(\Sigma_0)}^2  + s \mathfrak{h}_s(A^\pm,q^\pm,\Phi,\phi) \right),\ s \in (s_1,\infty).
\end{equation}
With reference to \eqref{def-theta}, we get by substituting the right hand side of \eqref{def-v0pm} for $v_0^\pm$ in \eqref{e11}, that
\begin{align}
\label{e12}
&  \sum_{\ka=+,-} \left( \norm{e^{s \al_0} \left( 2 i A^\ka \cdot \nabla u_0^\ka - \left(  \mathcal{S}^\ka \cdot A^\ka  - i ( \nabla \cdot A^\ka) + q^\ka \right) u_0^\ka -  \left( \ka \Phi \cdot \nabla u_0^{-\ka} + \phi u_0^{-\ka} \right)  \right)}_{L^2(\Om)}^2 \right) \\
\le & C  \left(  \sum_{\ka=+,-} \norm{\pa_\nu v^\ka}_{L^2(\Sigma_0)}^2 + s^{-\frac{1}{2}}  \mathfrak{h}_s(A^\pm,q^\pm,\Phi,\phi) \right),\ s \in (s_1,\infty). \nonumber 
\end{align}
Here $s_1$ was possibly replaced by $\max(s_1,1)$ and the notation $-\ka$ means $\mp$ whenever $\ka=\pm$.

The last step of the proof is to stably reconstruct the unknown coefficients $A^\pm$, $q^\pm$, $\Phi$ and $\phi$ by suitably choosing the initial states $u_0^\pm$ in \eqref{e12}. Otherwise stated we shall probe the system \eqref{sys:ori} with sufficiently many initial states $u_0^k=(u_0^{+,k},u_0^{-,k})$, $k=1,\cdots,3d+2$, and Dirichlet boundary conditions $g^k=(g^{+,k},g^{-,k})$ satisfying the compatibility condition \eqref{d4} with $m=4$, in order to extract the relevant information given by the estimate \eqref{e12} on $A^\pm$, $q^\pm$, $\Phi$ and $\phi$.

\subsection{End of the proof}
 \label{sec-eop}
Let us denote by $u^k=(u^{+,k},u^{-,k})$ the solution to \eqref{sys:ori} with initial state $(u_0^\pm,g^\pm)=(u_0^{\pm,k},g^{\pm,k})$. Set $v^{\pm,k}=\partial_t u^{\pm,k}$ and $\mu_k=\mu_k^+ +\mu_k^-$, where $\mu_k^\pm=\left\|   \pa_\nu v^{\pm,k} \right\|_{L^2(\Sigma_0)}^2$.

We proceed in three steps: \\
\noindent {\it Step 1.}  First we take $u_0^{+,1}=1$, $u_0^{-,1}=0$ and $u_0^{+,2}=0$, $u_0^{-,2}=1$ successively in \eqref{e12}, add the two obtained inequalities, and get that:
\begin{align}
\label{e21}
& \norm{e^{s \al_0} \phi}_{L^2(\Om)}^2 + \sum_{\ka=+,-} \norm{e^{s \al_0} \left(\cS^\ka \cdot A^\ka -i\na \cdot A^\ka +q^\ka\right)}_{L^2(\Om)}^2 \\
& \le  C  \left(s^{-\frac{1}{2}} \mathfrak{h}_s(A^\pm,q^\pm,\Phi,\phi)+\mu_1+\mu_2 \right),\ s \in (s_1,\infty). \nonumber
\end{align}

\noindent {\it Step 2}: Second, we pick $6d$ functions $u_0^{\pm,k+2}:\Om \longrightarrow \R$, $k=1,\cdots,3d$, such that the two matrices $(U_0^\pm)^T U_0^\pm$, where $U_0^\pm=\left(\partial_l u_0^{\pm,k+2}\right)_{1\leq k,l\leq 3d}$ and $(U_0^\pm)^{T}$ is the transpose of $U_0^{\pm}$, are strictly positive definite:
\begin{equation}
\label{e:4}
\exists \upsilon_0^\pm >0,\ \abs{U_0^\pm \xi} \ge \upsilon_0^\pm \abs{\xi},\ \xi \in \Bbb{C}^d.
\end{equation}
Then, for all $k=1,\cdots,3d$, we obtain upon substituting $u_0^{\pm,k+2}$ for $u_0^\pm$ in \eqref{e12} that
\begin{align}
\label{e22}
\sum_{\ka=+,-} \norm{\xi^{\ka,k+2} - i \zeta^{\ka,k+2}}_{L^2(\Om)}^2
& \le  C  \left(s^{-\frac{1}{2}}  \mathfrak{h}_s(A^\pm,q^\pm,\Phi,\phi) + \mu_{k+2}  \right),\ s \in (s_1,\infty),
\end{align}
where
\begin{equation}
\label{e23} 
\xi^{\ka,k+2}=e^{s \al_0}\left(-2A^\pm \cdot \na u_0^{\ka,k+2} + i \ka \Phi \cdot \na u_0^{-\ka,k+2}\right)
\end{equation}
and
\begin{equation}
\label{e24} 
\zeta^{\ka,k+2}= e^{s \al_0} \left( \left( \cS^\ka \cdot A^\ka -i \na \cdot A^\ka + q^\ka \right) u_0^{\ka,k+2} + \phi u_0^{-\ka,k+2} \right).
\end{equation}
Upon using that
\begin{align}
\label{e:4b}
\abs{\xi+\zeta}^2 \ge \f{1}{2} \abs{\xi}^2 - \abs{\zeta}^2,\  \xi, \zeta \in \Bbb{C}^d,
\end{align}
we infer from \eqref{e22} that
\begin{align}
\label{e25}
\sum_{\ka=+,-} \norm{\xi^{\ka,k+2}}_{L^2(\Om)}^2 
& \le  C  \left( s^{-\frac{1}{2}} \mathfrak{h}_s(A^\pm,q^\pm,\Phi,\phi) + \mu_{k+2} \right) + 2 \sum_{\ka=+,-} \norm{\zeta^{\ka,k+2}}_{L^2(\Om)}^2. 
\end{align}
Further, we have
\begin{align}
\label{e26}
\norm{\zeta^{\pm,k+2}}_{L^2(\Om)}^2 & \le C \left(  \norm{e^{s \al_0} \left( \cS^\pm \cdot A^\pm -i \na \cdot A^\pm + q^\pm \right)}_{L^2(\Om)}^2 +  \norm{e^{s \al_0} \phi}_{L^2(\Om)}^2 \right) \\
& \le C \left( s^{-\frac{1}{2}} \mathfrak{h}_s(A^\pm,q^\pm,\Phi,\phi)+\mu_1+\mu_2 \right), \nonumber
\end{align}
by \eqref{e21} and \eqref{e24},  whereas \eqref{e23} yields
$$  \norm{\xi^{\pm,k+2}}_{L^2(\Om)}^2= 4 \norm{e^{s \al_0} A^\pm \cdot \na u_0^{\pm,k+2}}_{L^2(\Om)}^2 +
\norm{e^{s \al_0} \Phi \cdot \na u_0^{\mp,k+2}}_{L^2(\Om)}^2, $$
since $A^\pm$ and $\Phi$ are real-valued. 
Inserting this and \eqref{e26} into \eqref{e25}, we get for all $k=1,\ldots,3d$, 
\begin{align*}
& \sum_{\ka=+,-} \left( \norm{e^{s \al_0} A^\ka \cdot \na u_0^{\ka,k+2}}_{L^2(\Om)}^2 + \norm{e^{s \al_0} \Phi \cdot \na u_0^{\ka,k+2}}_{L^2(\Om)}^2 \right) \\
& \le  C  \left( s^{-\frac{1}{2}} \mathfrak{h}_s(A^\pm,q^\pm,\Phi,\phi) + \mu_1 + \mu_2 + \mu_{k+2}  \right),\ s \in (s_1,\infty).
\end{align*}
Summing up the above estimate over $k=1,\cdots,3d$ and applying \eqref{e:4}, we find that
\begin{align}
\label{e27}
& \sum_{\ka=+,-} \norm{e^{s \al_0} A^\ka}_{L^2(\Om)^d}^2 + \norm{e^{s \al_0} \Phi}_{L^2(\Om)^d}^2 \\
& \le  C  \left( s^{-\frac{1}{2}} \mathfrak{h}_s(A^\pm,q^\pm,\Phi,\phi) + \sum_{k=1}^{3d+2} \mu_k  \right),\ s \in (s_1,\infty). \nonumber 
\end{align}

\noindent {\it Step 3.} 
Last, we combine \eqref{e:4b} where
$\xi=e^{s \al_0}(q^\pm-i \na \cdot  A^\pm)$ and $\zeta=e^{s\al_0} \cS^\pm \cdot A^\pm$, with the estimate $\norm{e^{s\al_0} \cS^\pm \cdot A^\pm}_{L^2(\Om)} \le 2M \norm{e^{s\al_0}A^\pm}_{L^2(\Om)^d}$, we get 
\begin{align*}
& \norm{e^{s\al_0}(q^\pm -i \na \cdot A^\pm)}_{L^2(\Om)}^2 \\
& \le  C \left( \norm{e^{s\al_0} A^\pm}_{L^2(\Om)^d}  + \norm{e^{s\al_0}(\cS^\pm \cdot A^\pm -i \na \cdot A^\pm + q^\pm)}_{L^2(\Om)}^2 \right),\ s \in (s_1,\infty). \nonumber
\end{align*}
This, \eqref{e21}, \eqref{e27} and the following identity
$$\norm{e^{s\al_0}(q^\pm -i \na \cdot  A^\pm)}_{L^2(\Om)}^2 = \norm{e^{s\al_0}q^\pm}_{L^2(\Om)}^2 + \norm{e^{s\al_0} \na \cdot A^\pm}_{L^2(\Om)}^2, $$
arising from the assumption that $A^\pm$ and $q^\pm$ are real-valued, yield
\begin{align*}
\norm{e^{s\al_0}q^\pm}_{L^2(\Om)}^2 + \norm{e^{s\al_0} \na \cdot A^\pm}_{L^2(\Om)}^2
	 \le  C  \left(s^{-\frac{1}{2}} \mathfrak{h}_s(A^\pm,q^\pm,\Phi,\phi) +\sum_{k=1}^{3d+2}\mu_{k}  \right),\ s \in (s_1,\infty).
\end{align*}
Putting this together with \eqref{e10}, \eqref{e21} and \eqref{e27}, we find
$$
\mathfrak{h}_s(A^\pm,q^\pm,\Phi,\phi) \le  C  \left(s^{-\frac{1}{2}} \mathfrak{h}_s(A^\pm,q^\pm,\Phi,\phi) +\sum_{k=1}^{3d+2}\mu_{k}  \right),\ s \in (s_1,\infty),
$$
which entails that
\begin{equation}
\label{e30}
\mathfrak{h}_s(A^\pm,q^\pm,\Phi,\phi) \le  C  \sum_{k=1}^{3d+2}\mu_{k},\ s \in (s_2,\infty),
\end{equation}
where $s_2 \in (s_1,\infty)$ is taken so large that $1- Cs_2^{-\frac{1}{2}}>\frac{1}{2}$.
Finally, recalling \eqref{e21} and using that
$$
e^{s\al_0(x)}=e^{s\ell^{-2}(0)(e^{\la\be(x)}-e^{\la K})}\ge e^{s\ell^{-2}(0)(1-e^{\la K})},\ x \in \Om,\ s \in (0,\infty),
$$
in accordance with \eqref{con:l}-\eqref{con:weight}, we end up getting the desired result directly from \eqref{e30}.

\subsection{Proof of Lemma \ref{lm0}}
\label{sec-prlm0}
With reference to \eqref{def-v0pm} and \eqref{def-V}-\eqref{def-T}, we have
$$ \norm{e^{s \alpha_0} v_0^\pm}_{L^2(\Om)} \le V_{\cS^\pm}(A^\pm,q^\pm,u_0^\pm) + T(\pm \Phi,\phi,u_0^\mp),\ s \in (0,\infty), $$
with
\begin{align*}
& V_{\cS^\pm}(A^\pm,q^\pm,u_0^\pm) \\
\le & 2 \norm{e^{s \alpha_0} A^\pm}_{L^2(\Om)^d} \left( \norm{\na u_0^\pm}_{L^\infty(\Om)^d} + M \norm{u_0^\pm}_{L^\infty(\Om)} \right) + \norm{e^{s \alpha_0} \na \cdot A^\pm}_{L^2(\Om)} \norm{u_0^\pm}_{L^\infty(\Om)} \\
& + \norm{e^{s \alpha_0} q^\pm}_{L^2(\Om)} \norm{u_0^\pm}_{L^\infty(\Om)}  \\
\le & C \norm{u_0^\pm}_{W^{1,\infty}(\Om)} \left( \norm{e^{s \alpha_0} A^\pm}_{L^2(\Om)^d} + 
\norm{e^{s \alpha_0} \na \cdot A^\pm}_{L^2(\Om)} + \norm{e^{s \alpha_0} q^\pm}_{L^2(\Om)} \right)
\end{align*}
and
\begin{align*}
T(\pm \Phi,\phi,u_0^\mp)
\le & \norm{e^{s \alpha_0} \Phi}_{L^2(\Om)^d} \norm{\na u_0^\mp}_{L^\infty(\Om)^d} + \norm{e^{s \alpha_0} \phi}_{L^2(\Om)} \norm{u_0^\mp}_{L^\infty(\Om)} \\
\le & \norm{u_0^\mp}_{W^{1,\infty}(\Om)} \left(  \norm{e^{s \alpha_0} \Phi}_{L^2(\Om)^d} + \norm{e^{s \alpha_0} \phi}_{L^2(\Om)} \right).
\end{align*}
This entails that
\begin{equation}
\label{e100}  
\norm{e^{s \alpha_0} v_0^\pm}_{L^2(\Om)} \le C \left(  \norm{u_0^\pm}_{W^{1,\infty}(\Om)}  +\norm{u_0^\mp}_{W^{1,\infty}(\Om)} \right) \mathfrak{h}_s(A^\pm,q^\pm,\Phi,\phi)^{1 \slash 2},\ s \in (0,\infty).
\end{equation}

Let us denote by $\mathbb{J}_{X}=\left( \partial_i x_j \right)_{1 \le i, j \le n}$ the Jacobian matrix of $X=(x_j)_{1 \le j \le n} \in H^1(\Om,\R^d)$ and by $\mathbb{D}^2_{u_0^\pm}=\left( \partial_{i j}^2 u_0^\pm \right)_{1 \le i, j \le n}$ the Hessian matrix of $u_0^\pm$. Then, with reference to \eqref{def-theta} we get through direct computation that
\begin{align}
\label{e101}
\na \left( \vartheta_{\cS^\pm}(A^\pm,q^\pm) u_0^\pm \right) 
= & 2 i \left( \mathbb{J}_{A^\pm} \na u_0^\pm + \mathbb{D}_{u_0^\pm}^2 A^\pm \right) + \left( i \na \cdot A^\pm - \cS^\pm \cdot A^\pm - q^\pm \right) \na u_0^\pm  \\
& + u_0^\pm \left( i \na (\na \cdot A^\pm) - \mathbb{J}_{\cS^\pm} A^\pm -  \mathbb{J}_{A^\pm} \cS^\pm - \na q^\pm \right)\nonumber
\end{align}
and that
\begin{equation}
\label{e102}
\na \left( \Theta(\pm \Phi,\phi) u_0^\mp \right) 
= \pm \mathbb{J}_{\Phi} \na u_0^\mp \pm \mathbb{D}_{u_0^\mp}^2 \Phi + u_0^\mp \na \phi  + \phi \na u_0^\mp.
\end{equation}
Further since
$$ \na v_0^\pm = i \left( \na \left( \vartheta_{\cS^\pm}(A^\pm,q^\pm) u_0^\pm \right)  - \na \left( \Theta(\pm \Phi,\phi) u_0^\mp \right) \right), $$
according to \eqref{def-v0pm}, we deduce from \eqref{e101}-\eqref{e102} that
\begin{equation}
\label{e103}
\na v_0^\pm = -i \left( \mathfrak{f}^\pm + \mathfrak{g}^\pm \right),
\end{equation}
where
\begin{equation}
\label{e104}
\mathfrak{f}^\pm = -2 i \mathbb{D}_{u_0^\pm}^2 A^\pm + \left( \cS^\pm \cdot A^\pm -i \na \cdot A^\pm + q^\pm \right) \na u_0^\pm + u_0^\pm \mathbb{J}_{\cS^\pm} A^\pm   \pm \mathbb{D}_{u_0^\mp}^2 \Phi + \phi \na u_0^\mp
\end{equation}
and
\begin{equation}
\label{e105}
\mathfrak{g}^\pm = -2 i \mathbb{J}_{A^\pm} \na u_0^\pm + u_0^\pm \left( \mathbb{J}_{A^\pm} \cS^\pm -i \na (\na \cdot A^\pm) + \na q^\pm \right)  \pm \mathbb{J}_{\Phi} \na u_0^\mp + u_0^\mp \na \phi.
\end{equation}
Moreover, we have
\begin{align}
\label{e107}
\norm{e^{s \alpha_0} \mathfrak{f}^\pm}_{L^2(\Om)^d} 
\le & C \left( \norm{u_0^\pm}_{W^{2,\infty}(\Om)} + \norm{u_0^\mp}_{W^{2,\infty}(\Om)} \right) \mathfrak{h}_s(A^\pm,q^\pm,\Phi,\phi)^{1 \slash 2} 
\end{align}
from \eqref{e104}
and
\begin{align*}
& \norm{e^{s \alpha_0} \mathfrak{g}^\pm}_{L^2(\Om)^d} \\
\le & C \left( \norm{u_0^\pm}_{W^{1,\infty}(\Om)} \left( \norm{e^{s \alpha_0} \mathbb{J}_{A^\pm}}_{L^2(\Om)^{n^2}}+ \norm{e^{s \alpha_0}  \na (\na \cdot  A^\pm)}_{L^2(\Om)^d} +
\norm{e^{s \alpha_0} \na q^\pm}_{L^2(\Om)^d} \right) \right. \nonumber \\
& \left. + \norm{u_0^\mp}_{W^{1,\infty}(\Om)} \left( \norm{e^{s \alpha_0} \mathbb{J}_{\Phi}}_{L^2(\Om)^{n^2}} + \norm{e^{s \alpha_0} \na \phi}_{L^2(\Om)^d} \right)\right)
\end{align*}
from \eqref{e105}. Using that
$$ \abs{\nabla (\nabla \cdot A^\pm)(x)} + \max_{i=1,\ldots,d} \sum_{j=1}^d \left( \abs{\partial_i A_j^\pm(x)} + \abs{\partial_i \Phi_j^\pm(x)} \right)\le M 
\left(\abs{\nabla.A^\pm(x)}+ \abs{A^\pm(x)} + \abs{\Phi^\pm(x)} \right),\ x \in \Om $$
and
$$ \abs{\nabla q^\pm(x)}  + \abs{\nabla\phi^\pm(x)} \le M \left( \abs{q^\pm(x)} + \abs{\phi^\pm(x)} \right),\ x \in \Om, $$ 
by assumption, we infer from the above estimate that
\begin{align}
\label{e108}
\norm{e^{s \alpha_0} \mathfrak{g}^\pm}_{L^2(\Om)^d} 
\le & C \left( \norm{u_0^\pm}_{W^{1,\infty}(\Om)} + \norm{u_0^\mp}_{W^{1,\infty}(\Om)} \right) \mathfrak{h}_s(A^\pm,q^\pm,\Phi,\phi)^{1 \slash 2}.
\end{align}

Finally, keeping in mind that
$$
\ov{v_0^\pm} \nabla v_0^\pm -v_0^\pm \nabla \ov{v_0^\pm}=  2 i \left( \Im ( v_0^\pm \ov{\mathfrak{f}^\pm}) + \Im (v_0^\pm \ov{\mathfrak{g}^\pm}) \right),
$$
according to \eqref{e103}, the desired result follows from \eqref{defI} and the estimates \eqref{e100} and \eqref{e107}-\eqref{e108}.

\appendix

\section{The magnetic Dirichlet Laplacian}
For $A \in L^{\infty}(\Om,\R^d)$, we consider the sesquilinear form
$$ a(u,v) = \int_{\Om} \na_A u(x) \cdot \overline{\na_A v(x)} dx,\ u,v \in H_0^1(\Om), $$
where we recall that $\nabla_A$ stands for the magnetic gradient operator $\nabla + i A$, and we denote by $-\Delta_A^D$ the linear operator generated by $a$ in $L^2(\Om)$.

\begin{lem}
\label{lm-mdl}
The magnetic Dirichlet Laplacian $-\Delta_A^D$ is self-adjoint in $L^2(\Om)$. Moreover, when $A \in W^{1,\infty}(\Omega)$ the operator $-\Delta_A^D$ acts as $-\Delta_A$, defined in \eqref{def-deltaA}, on its domain
$$ D(-\Delta_A^D)=H_0^1(\Om) \cap H^2(\Om). $$
\end{lem}
\begin{proof}
For all $u$ and $v$ in $H_0^1(\Om$, we have
$$ \abs{a(u,v)} \le 2(1+\norm{A}_{L^\infty(\Om)^d}^2)\norm{u}_{H^1(\Om)} 
\norm{v}_{H^1(\Om)},$$
hence $a$ is continuous on $H_0^1(\Om)$. Further, for all $u \in H_0^1(\Om)$ such that
$v \mapsto a(u,v)$ is continuous in $H_0^1(\Om)$ for the usual topology of $L^2(\Om)$, 
we set $-\Delta_A^D u = f_u$, where $f_u$ is the unique vector in $L^2(\Om)$ given by the Riesz representation theorem, such that
$$ a(u,v)=\langle f_u , v \rangle_{L^2(\Om)},\ v \in H_0^1(\Om). $$

Moreover, since
\begin{align*}
\norm{\nabla_A u}_{L^2(\Om)^d}^2 & \ge \norm{\nabla u}_{L^2(\Om)^d}^2
+ \norm{A u}_{L^2(\Om)^d}^2 - 2 \norm{\nabla u}_{L^2(\Om)^d} \norm{A u}_{L^2(\Om)^d} \\
& \ge \frac{\norm{\nabla u}_{L^2(\Om)^d}^2}{2} - \norm{A u}_{L^2(\Om)^d}^2,\ u \in H_0^1(\Om),
\end{align*}
the sesquilinear form $a$ is $H^1(\Om)$-elliptic with respect to $L^2(\Om)$, in the sense that we have
$$ 
a(u,u) + \lambda \norm{u}_{L^2(\Om)}^2 \ge \frac{1}{2} \norm{u}_{H^1(\Om)}^2,\ u \in H_0^1(\Om),
$$
with $\lambda=\norm{A}_{L^\infty(\Om)^d}^2 + \frac{1}{2}$. Therefore, $-\Delta_A^D$ is densely defined in $L^2(\Omega)$. Finally, since $\Gamma$ is $C^2$ the domain of $-\Delta_A^D$ is $H_0^1(\Om) \cap H^2(\Om)$, see e.g. \cite[Section 2]{C}.
\end{proof}

\section{Relatively bounded perturbation}
The following technical result establishes that first order differential operators are relatively bounded with respect to the magnetic Dirichlet Laplacian.

\begin{lem}
\label{lm-rbp}
Let $A \in L^\infty(\Om,\R^d)$. Then for all $\Phi \in L^\infty(\Om,\R^d)$, the operator $\Phi \cdot \nabla$ is $\Delta_A^D$-bounded with relative bound zero.
\end{lem}
\begin{proof}
Using that
$$ \Phi \cdot \nabla u = \Phi \cdot \nabla_A u - i (\Phi \cdot A) u,\ u \in H^1(\Om), $$
we get for all $u \in H_0^1(\Om) \cap H^2(\Om)$ that
\begin{align*}
\norm{\Phi \cdot \nabla u}_{L^2(\Om)}^2 & \le 2 \left( \norm{\Phi}_{L^\infty(\Om)^d}^2 \norm{\nabla_A u}_{L^2(\Om)^d}^2 + \norm{\Phi \cdot A}_{L^\infty(\Om)}^2 \norm{u}_{L^2(\Om)}^2 \right) \\
& \le 2 \left( \norm{\Phi}_{L^\infty(\Om)^d}^2 \langle -\Delta_A^D u , u \rangle_{L^2(\Om)} + \norm{\Phi}_{L^\infty(\Om)^d}^2  \norm{A}_{L^\infty(\Om)^d}^2 \norm{u}_{L^2(\Om)}^2 \right),
\end{align*}
where $\langle \cdot , \cdot \rangle_{L^2(\Om)}$ denotes the usual scalar product in $L^2(\Om)$. 
Therefore, by applying successively the Cauchy-Schwarz and Young inequalities, we find that
\begin{align*}
\norm{\Phi \cdot \nabla u}_{L^2(\Om)}^2
& \le 2 \left( \norm{\Phi}_{L^\infty(\Om)^d}^2 \norm{\Delta_A^D u}_{L^2(\Om)} \norm{u}_{L^2(\Om)} + \norm{\Phi}_{L^\infty(\Om)^d}^2  \norm{A}_{L^\infty(\Om)^d}^2 \norm{u}_{L^2(\Om)}^2 \right) \\
& \le \epsilon \norm{\Delta_A^D u}_{L^2(\Om)}^2 +  C_\epsilon \norm{u}_{L^2(\Om)}^2,\ \epsilon \in (0,1),
\end{align*}
where $C_\epsilon=\norm{\Phi}_{L^\infty(\Om)^d}^2 \left( \epsilon^{-1} \norm{\Phi}_{L^\infty(\Om)^d}^2 + 2  \norm{A}_{L^\infty(\Om)^d}^2  \right)$,
which entails the result.
\end{proof}

\vspace*{1cm}
\end{document}